\documentclass[10pt,final,leqno,onefignum,onetabnum,hidelinks]{siamltex1213}
\usepackage{amssymb,amsfonts,amsmath}
\usepackage{hyperref}
\usepackage{cite}

\newtheorem{remark}{Remark}[section]
\newtheorem{example}{Example}[section]
\newcommand{\AC}{\mathcal{O}\!\!\iota}
\newcommand{\one}{1\!{\mathrm l}}

\def\refe#1{{\rm(\ref{#1})}}
\def\iu{{\rm i}}

\def\Z{{\mathbb Z}}
\def\N{{\mathbb N}}
\def\R{{\mathbb R}}

\def\T{{\mathbb T}}

\def\eps{\varepsilon}

\newcommand\andquad{\quad\hbox{ and}\quad }
\renewcommand\Re{\mathrm{Re}\,}

\newcommand{\nn}{\nonumber}

\def\C{\mathbb{C}}
\def \d {\mathrm{d}}
\newcommand{\inv}{^{-1}}

\def \to {\rightarrow}

\begin{document}

\title{
A-stable time discretizations preserve \\
maximal parabolic regularity
}
\author{Bal\'{a}zs~Kov\'{a}cs\thanks{Mathematisches Institut,
Universit\"at T\"ubingen, Auf der Morgenstelle 10,
D-72076 T\"ubingen, Germany.
\email{Kovacs@na.uni-tuebingen.de} }
\and Buyang~Li\thanks{Department of
Mathematics, Nanjing University,
210093 Nanjing, P.R. China.
\email{buyangli@nju.edu.cn}}
\and Christian~Lubich\thanks{Mathematisches Institut,
Universit\"at T\"ubingen, Auf der Morgenstelle 10,
  D-72076 T\"ubingen, Germany.
  \email{Lubich@na.uni-tuebingen.de}}
}

\date{Received: \today }


\date{}

\maketitle
\begin{abstract} It is shown that for a parabolic problem with maximal $L^p$-regularity (for $1<p<\infty$), the time discretization by a linear multistep method or Runge--Kutta method has maximal $\ell^p$-regularity uniformly in the stepsize if the method is A-stable (and satisfies minor additional conditions). In particular, the implicit Euler method, the Crank-Nicolson method, the second-order backward difference formula (BDF), and the Radau IIA and Gauss Runge--Kutta methods of all orders preserve maximal regularity. The proof uses Weis' characterization of maximal $L^p$-regularity in terms of $R$-boundedness of the resolvent, a discrete operator-valued Fourier multiplier theorem by Blunck, and generating function techniques that have been familiar in the  stability analysis of time discretization methods since the work of Dahlquist. The A($\alpha$)-stable higher-order BDF methods have maximal $\ell^p$-regularity under an $R$-boundedness condition in a larger sector.
As an illustration of the use of maximal regularity in the error analysis of discretized nonlinear parabolic equations, it is shown how error bounds are obtained without using any growth condition on the nonlinearity or for nonlinearities having singularities.

\end{abstract}

\section{Introduction}
\setcounter{equation}{0}

Maximal regularity is an important
mathematical tool in studying existence, uniqueness
and regularity of the solution of nonlinear parabolic
partial differential equations (PDEs)
\cite{Amann,KW04,LSU68,Lions,Lunardi}.
A generator $A$ of an analytic semigroup
on a Banach space $X$ is said to have
{\it maximal $L^p$-regularity} if the 
solution of the evolution equation
\begin{equation} 
\label{IVP}
    \left\{
    \begin{aligned}
         u'(t) =&\, Au(t) + f(t), \qquad t>0,\\
        u(0) =&\, 0,
    \end{aligned}
    \right.
\end{equation}
satisfies
\begin{align}\label{MaxP}
\|u'\|_{L^p(\R_+;X)}
+\|Au\|_{L^p(\R_+;X)}\leq C_{p,X}
\|f\|_{L^p(\R_+;X)}
\quad
\forall\, f\in L^p(\R_+;X) 
\end{align} 
for some (or, as it turns out, for all) $1<p<\infty$. 
On a Hilbert space, every generator of a bounded analytic semigroup
has maximal $L^p$-regularity \cite{deSimon}, and
Hilbert spaces are the only Banach spaces for which this holds true \cite{KL00}.
Beyond Hilbert spaces, a characterization of the maximal $L^p$-regularity
was given by Weis \cite{Weis1,Weis2} on $X=L^q(\Omega)$ (with $1<q<\infty$ and $\Omega$ a region in $\R^d$) and more generally on UMD spaces
in terms of the $R$-boundedness of the resolvent operator. Operators having maximal $L^p$-regularity include
elliptic differential operators on $L^q(\Omega)$ with general boundary conditions, and
operators that generate a positive and contractive semigroup on $L^q(\Omega, \d\mu)$ spaces for an arbitrary measure space $(\Omega,\d\mu)$, as do many generators of stochastic processes;
see \cite{KW04} and references therein.

In this paper we address the following question: {\it Given an operator $A$ that has maximal $L^p$-regularity, for which (if any) time discretization methods for \eqref{IVP} is the maximal $L^p$-regularity preserved in the discrete $\ell^p$-setting, uniformly in the stepsize?}

We will show that this holds for {\it A-stable} multistep and Runge--Kutta methods, under minor additional conditions.
In particular, the implicit Euler method, the Crank--Nicolson method, the second-order backward difference formula (BDF), and higher-order A-stable implicit Runge--Kutta methods such as the Radau IIA and Gauss methods all preserve maximal regularity.

We recall that a numerical time discretization method is called A-stable if for every complex $\lambda$ with $\Re\lambda\le 0$, for every stepsize $\tau>0$, and for arbitrary starting values, the numerical solution of the scalar linear differential equation $y'=\lambda y$ remains bounded as the discrete time goes to $+\infty$. It is remarkable that this deceivingly simple and well-studied concept, which was introduced by Dahlquist \cite{Dahlquist63}, essentially suffices to yield maximal $\ell^p$-regularity, uniformly in the stepsize for every operator $A$ that has maximal $L^p$-regularity \eqref{MaxP}, both on Hilbert spaces and on a large class of Banach spaces.

Our proofs rely on Weis' characterization of maximal $L^p$-regularity on UMD spaces \cite{Weis1}, on a discrete operator-valued Fourier multiplier theorem of Blunck \cite{Blunck01}, and on generating function techniques for time discretization methods, which have been familiar for linear multistep methods since the work of Dahlquist \cite{Dahlquist,Dahlquist63}, but are less common for Runge--Kutta methods \cite{LO93}.

A characterization of discrete maximal $\ell^p$-regularity
for recurrence relations $u_{n+1}=Tu_n+f_n$ ($n\ge 0$)
was given in \cite{Blunck01,Blunck01-2},
and a generalization of these results to the explicit Euler scheme with certain non-constant time step sequences
was given in \cite{Portal}; also see \cite{ACL}. We are not aware, however, of previous results of discrete maximal $\ell^p$-regularity, uniformly in the stepsize and the number of time steps, for even just the implicit Euler method, let alone for other A-stable time discretizations as studied in this paper.

Maximal $L^p$-regularity of finite element spatial semi-discretizations of parabolic \linebreak PDEs has been used
in the analysis of numerical methods for PDEs
with minimal regularity assumption on the solution
\cite{Gei1,Gei2,Li} or on the diffusion coefficient \cite{LS15}.
In order to prove the convergence of fully discrete solutions
of some nonlinear PDEs,
e.g., the dynamic Ginzburg--Landau equations \cite{Li-GL},
maximal $L^p$-regularity in the
time-discrete setting as given here is needed. 

The paper is organized as follows.

In Section~\ref{section: preliminaries} we recall important notions and results from the theory of maximal parabolic regularity: Weis' characterization of maximal $L^p$-regularity on UMD spaces, $R$-boundedness, and operator-valued Fourier multipliers in a Banach space setting \cite{Weis1,Weis2,Blunck01,KW04}.

In Section~\ref{section: IE and CN} we show discrete maximal $\ell^p$-regularity estimates for two simple one-step methods, the backward Euler method and the Crank--Nicolson scheme. This allows us to show the basic arguments in a technically simpler setting than for the other methods considered in later sections.

In Sections \ref{section: BDF} and \ref{section: RK}  maximal $\ell^p$-regularity results are shown for higher order methods, backward difference formulae (BDF) and A-stable Runge--Kutta methods, respectively. While linear multistep methods have a scalar differentiation symbol in the appearing generating functions, the differentiation symbol of Runge--Kutta methods is matrix-valued, which makes the analysis more complicated.

In Section \ref{section: full discretization} we briefly discuss maximal regularity of full discretizations and show how uniformity of the bounds in both the spatial gridsize $h$ and the temporal stepsize $\tau$ can be obtained.

In Section~\ref{section: ell infty est} we give $\ell^\infty$-bounds that show maximal regularity up to a factor that is logarithmic in the number of time steps considered. These bounds are obtained for a subclass of methods that includes the BDF and Radau IIA methods, but not the Crank-Nicolson and Gauss methods. The $\ell^\infty$-bounds rely on the convolution quadrature interpretation of linear multistep methods \cite{Lubich-cqrevisited} and Runge--Kutta methods \cite{LO93}. A related result has recently been proved for discontinuous Galerkin time-stepping methods in \cite{LB15}, using different techniques.

Finally, in Section~\ref{section: applications} we
illustrate the use of discrete maximal $\ell^p$-regularity in deriving
error bounds for discretizations of nonlinear parabolic PDEs. We show that in contrast to previously existing techniques, the approach via discrete maximum regularity enables us to obtain optimal-order error bounds without any growth condition on the pointwise nonlinearity $f(u,\nabla u)$ and also for nonlinearities having singularities. This becomes possible because via the discrete maximal $\ell^p$-regularity we can
control the $\ell^\infty(W^{1,\infty})$-norm of the error, provided the exact solution of the parabolic problem has sufficient regularity.

\section{Preliminaries}
\label{section: preliminaries}

Here we collect basic results on maximal $L^p$-regularity and related concepts, which will be needed later on. For further background and details, proofs and references we
 refer to the excellent lecture notes by Kunstmann \& Weis \cite{KW04}.

 \subsection{Characterization of maximal $L^p$-regularity in terms of the resolvent}

As was shown by Weis \cite{Weis2}, maximal $L^p$-regularity of an operator $A$ on a Banach space $X$ can be characterized in terms of its resolvent $(\lambda-A)^{-1}=R(\lambda,A)$ for a large class of Banach spaces that includes Hilbert spaces and $L^q(\Omega,\d\mu)$-spaces with $1<q<\infty$. We begin with formulating the notions that permit us to state this fundamental result.

A Banach space $X$ is said to be a {\it UMD space}
if the Hilbert transform
\begin{align*}
Hf(t)={\rm P.V.}\int_{\R}\frac{1}{t-s}f(s)\, \d s
\end{align*}
is bounded on $L^p(\R;X)$ for all $1<p<\infty$; see \cite{KW04}.
From \cite{Burkholder81,Bourgain83}
we know that this definition is equivalent to
the definition by using the
unconditional martingale differences approach, which explains the abbreviation UMD.
Examples of UMD spaces include Hilbert spaces and
$L^q(\Omega,\d\mu)$
and its closed subspaces,
where $(\Omega,\d\mu)$ is any measure space
and $1<q<\infty$. Throughout the paper $X$ always denotes a UMD space, unless otherwise stated.

A collection of operators $\{M(\lambda): \lambda\in \Lambda\}$
is said to be {\it $R$-bounded}
if there is a positive constant $C_R$,
called the $R$-bound of the collection, such that
any finite subcollection of operators
$M(\lambda_1), M(\lambda_2), \dotsc, M(\lambda_l)$
satisfies
\begin{align*}
\int_0^1\bigg\|\sum_{j=1}^lr_j(s)M(\lambda_j)v_j\bigg\|_X^2\d s
\leq C_R^2\int_0^1\bigg\|\sum_{j=1}^lr_j(s)v_j\bigg\|_X^2\d s,
\quad
\forall\,\, v_1,v_2,...,v_l\in X,
\end{align*}
where $r_j(s) = \hbox{sign}\phantom{.}\sin(2j\pi s)$, for $j=1,2, ...$, are the Rademacher functions
defined on the interval $[0,1]$.

In the special case $X=L^q(\Omega,\d\mu)$ a simpler condition suffices:
there, a collection of operators $\{M(\lambda): \lambda \in \Lambda\}$
is $R$-bounded if and only if there is a positive constant $C_R^*$ such that
any finite subcollection of operators
$M(\lambda_1)$, $M(\lambda_2)$, $\dotsc$, $M(\lambda_l)$
satisfies
\begin{align*}
\bigg\|\bigg(\sum_{j=1}^l
|M(\lambda_j)v_j|^2\bigg)^{\frac{1}{2}}\bigg\|_{L^q}
\! \! \! \leq C_R^*\bigg\|\bigg(\sum_{j=1}^l
|v_j|^2\bigg)^{\frac{1}{2}}\bigg\|_{L^q} ,
\quad
\forall\,\, v_1,v_2,...,v_l\in L^q(\Omega,\d\mu) .
\end{align*}
For Hilbert spaces, a collection of operators is $R$-bounded if and only if it is bounded, and the $R$-bound equals its bound.

We can now state Weis' characterization of maximal $L^p$-regularity. Here
$\Sigma_\vartheta$ denotes the sector 
$\Sigma_\vartheta = \{ z\in \C\backslash\{0\}: |\arg z|<\vartheta\}$.

\begin{theorem}[Weis \cite{Weis2}, Theorem 4.2]
\label{MaxR}
Let $X$ be a UMD space and let $A$ be the generator of
a bounded analytic semigroup on $X$.
Then $A$ has maximal $L^p$-regularity
if and only if for some $\vartheta>\pi/2$ the set of operators
          $
                \{ \lambda (\lambda - A)\inv \,:
                \, \lambda \in \Sigma_\vartheta \}
         $ 
is $R$-bounded.
\end{theorem}




\subsection{Operator-valued multiplier theorems}
The ``if" direction of Theorem~\ref{MaxR} is obtained from the following result, which extends a scalar-valued Fourier multiplier theorem of Mikhlin.  Here, $\mathcal{F}$ denotes the Fourier transform on $\R$: for appropriate $f$,
$$
\mathcal{F}f(\xi) =  \int_\R e^{-i\xi t} f(t)\,dt, \qquad \xi\in \R.
$$
$B(X)$ denotes the space of bounded linear operators on $X$.

\begin{theorem}[Weis \cite{Weis2}, Theorem 3.4]
\label{thm:weis-mult}
Let $X$  be a UMD space. Let $M: \R \setminus \{0\} \to B(X )$ be a differentiable function such that the set
$$
\{M(\xi)\,:\,\xi \in\R\setminus\{0\}\} \ \cup\ 
\{\xi M'(\xi)\,:\, \xi  \in\R\setminus\{0\}\} \quad\hbox{is $R$-bounded,}
$$
with $R$-bound $C_R$.
Then, $\mathcal{M} f = \mathcal{F}\inv\bigl( M(\cdot) (\mathcal{F}f)(\cdot)\bigr)$ extends to a bounded operator
$$
\mathcal{M} :
L^p(\R,X) \to L^p(\R,X) \quad\hbox{ for }\ 1 < p < \infty.
$$
Moreover, there exists a constant $C_{p,X}$ independent of $M$ such that the operator norm of
$\mathcal{M} $ is bounded by $C_{p,X}C_R$.
\end{theorem}

On noting that the resolvent is the Laplace transform of the semigroup, with $M(\xi)=i\xi (i\xi - A)\inv$ it is seen that $\mathcal{M}f = u'$, for the solution $u$ of \eqref{IVP}. By Theorem~\ref{thm:weis-mult}, $R$-boundedness of $\lambda (\lambda - A)\inv$ on the imaginary axis therefore yields maximal $L^p$-regularity of $A$.

In this paper we will use the discrete version of Theorem~\ref{thm:weis-mult}. Here, $\mathcal{F}$ denotes the Fourier transform on $\Z$, which maps a sequence to its Fourier series on the torus $\T=\R/{2\pi\Z}$: for appropriate $f=(f_n)_{n\in\Z}$,
$$
\mathcal{F}f(\theta) =  \sum_{n\in\Z} e^{i\theta n} f_n, \qquad \theta\in \T.
$$

\begin{theorem}[Blunck \cite{Blunck01}, Theorem 1.3]
\label{thm:Blunck}
Let $X$ be a UMD space. Let $\widetilde M: (-\pi,0)\cup(0,\pi) \to B(X)$ be a differentiable function such that the set
\begin{equation}\label{M-tilde-set}
\big\{\widetilde M(\theta)\,:\,\theta \in (-\pi,0)\cup(0,\pi)\big\} \ \cup\ 
\big\{(1-e^{i\theta})(1+e^{i\theta}) \widetilde M'(\theta)\,:\, \theta \in (-\pi,0)\cup(0,\pi)\big\}
\end{equation}
is $R$-bounded, with $R$-bound $C_R$.
Then, $\mathcal{M} f = \mathcal{F}\inv\bigl( \widetilde M(\cdot) (\mathcal{F}f)(\cdot)\bigr)$ extends to a bounded operator
$$
\mathcal{M} :
\ell^p(\Z,X) \to \ell^p(\Z,X) \quad\hbox{ for }\ 1 < p < \infty.
$$
Moreover, there exists a constant $C_{p,X}$ independent of $\widetilde M$ such that the operator norm of
$\mathcal{M} $ is bounded by $C_{p,X}C_R$.
\end{theorem}

We will encounter the situation where the generating function of a sequence $\{M_n\}_{n\ge 0}$ of operators on $X$ converges on the complex unit disk:
\begin{equation*}
M(\zeta) = \sum_{n=0}^\infty M_n \zeta^n, \qquad |\zeta|<1,
\end{equation*}
and the radial limits
\begin{equation}\label{M-tilde}
\widetilde M(\theta) = \lim_{r\nearrow 1} M(re^{i\theta})
\end{equation}
exist for $\theta\ne 0,\pi$ and satisfy the conditions  of Theorem~\ref{thm:Blunck}.
For a sequence $f=(f_n)_{n\ge 0}\in\ell^p(X):= \ell^p(\N,X)$, extended to negative subscripts $n$ by $0$, the operator $\mathcal{M}$ is then given by the discrete convolution
$$
(\mathcal{M}f)_n = \sum_{j=0}^n M_{n-j} f_j, \qquad n=0,1,2,\ldots
$$

\subsection{Enlarging $R$-bounded sets of operators}

By the definition of $R$-bounded\-ness, it is clear that
if $\{M_1(\lambda):\lambda\in \Lambda\}$ and
$\{M_2(\lambda):\lambda\in \Lambda\}$ are two $R$-bounded
collections of operators on $X$, then 
$\{M_1(\lambda)+M_2(\lambda):\lambda\in \Lambda\}$ and 
$\{M_1(\lambda)M_2(\lambda):\lambda\in \Lambda\}$
are also $R$-bounded. Moreover, the union of two $R$-bounded collections is
$R$-bounded, and  the closure of an $R$-bounded collection of operators 
in the strong topology of $B(X)$ is again $R$-bounded.

The following lemma is often used to
prove the $R$-boundedness of
a collection of operators.

\begin{lemma}[\!\!\cite{CPSW}, Lemma 3.2]
Let $\mathcal{T}$ be an $R$-bounded set of linear operators on $X$, with $R$-bound $C_R$. Then the absolute convex hull of $\mathcal{T}$, that is, the collection of all finite linear combinations of operators in $\mathcal{T}$ with complex coefficients whose absolute sum is bounded by $1$, is also $R$-bounded, with $R$-bound at most $2C_R$.
\end{lemma}
%

A simple consequence of this lemma
is the following, which we will use later.

\begin{lemma}
\label{lemma: abs convex hull}
    Let $\{M(z) \ : \ z\in{\Gamma}\} \subset B(X)$
 be a $R$-bounded collection of operators, with $R$-bound $C_R$,
 where $\Gamma$ is a contour in the complex plane.
 Let $f(\lambda,z)$ be a complex-valued function of $z\in \Gamma$
 and $\lambda\in \Lambda$.
If
\begin{equation*}
     \int_\Gamma|f(\lambda,z)|\cdot |\d z|\leq C_0,
\end{equation*}
where $C_0$ is independent of $\lambda\in \Lambda$, then
the collection of operators in the closure of the absolute convex hull
    of $\{M(z)  :  z\in{\Gamma}\} $
    in the strong topology of $B(X)$,
    \begin{equation*}
   \biggl\{   \frac 1{C_0} \int_\Gamma f(\lambda,z)M(z)\,\d z: \, \lambda\in \Lambda \biggr\},
    \end{equation*}
 is $R$-bounded
with $R$-bound at most $2C_R$.
\end{lemma}

\section{Implicit Euler and Crank--Nicolson method}
\label{section: IE and CN}
We first present basic ideas to prove discrete maximal parabolic regularity on two simple methods. Later these ideas will be carried over to higher-order BDF and Runge--Kutta methods, where the key properties remain $R$-boundedness and A- or A($\alpha$)-stability.

We consider the backward Euler and Crank--Nicolson method applied with stepsize $\tau>0$,
\begin{align}\label{Euler}
  \frac{u_n - u_{n-1}}{\tau} =&\ Au_n + f_n,  \qquad  & & n\geq1,  &u_0=0,
\intertext{and}
\label{C-N}
    \frac{u_n - u_{n-1}}{\tau} =&\ A \frac{u_n + u_{n-1}}{2} + \frac{f_n+f_{n-1}}2\,,  \qquad & & n\geq1,
    &u_0=0 .
\end{align}
In this section, we use the following
notation for the backward difference:
$$
    \dot u_n = \frac{u_n - u_{n-1}}{\tau}.
$$


\begin{theorem}
\label{theorem: IE}
 If 
    $A$ has maximal $L^p$-regularity, for $1<p<\infty$,
then the numerical solution $(u_n)_{n=1}^N$ of \eqref{Euler}, obtained by the backward Euler method with stepsize $\tau$, satisfies the discrete maximal regularity estimate
    \begin{equation*}
        \big\|(\dot u_n )_{n=1}^N\big\|_{\ell^p(X)} + \big\|(A u_n )_{n=1}^N\big\|_{\ell^p(X)} \leq C_{p,X} \big\|(f_n)_{n=1}^N\big\|_{\ell^p(X)},
    \end{equation*}
    where the constant is independent of $N$ and $\tau$.
\end{theorem}
\begin{proof} 
We use the generating functions $u(\zeta)=\sum_{n=0}^\infty u_n \zeta^n$ and $f(\zeta)=\sum_{n=0}^\infty f_n \zeta^n$. Since the initial value is zero, we obtain
\begin{equation*}
   \Bigl (\frac{1-\zeta}\tau -A\Bigr) u(\zeta) = f(\zeta)
\end{equation*}
 and hence $\dot u(\zeta)= \sum_{n=0}^\infty \dot u_n \zeta^n$ is given by
\begin{equation*}
  \dot u(\zeta)=  \frac{1-\zeta}{\tau} u(\zeta) =  M(\zeta)f(\zeta)
  \quad\hbox{ with } \quad
  M(\zeta)=
  \frac{1-\zeta}{\tau}\Bigl( \frac{1-\zeta}{\tau}-A\Bigr)\inv .
\end{equation*}

In view of Theorem~\ref{thm:Blunck},
we only have to show analyticity of $M(\zeta)$ in the open unit disk $|\zeta|<1$ and the $R$-boundedness of the
 set \eqref{M-tilde-set}, with $\widetilde M(\theta)=M(e^{i\theta})$ for $\theta\ne 0$ modulo $2\pi$. To this end we show that the set
 \begin{equation}\label{M-Euler-R-bounded}
 \big\{ M(\zeta)\,:\, |\zeta|\le 1, \zeta\ne 1\big\} \cup \big\{(1-\zeta)M'(\zeta)\,:\,|\zeta|\le 1, \zeta\ne 1\big\}
 \quad\hbox{is $R$-bounded,}
 \end{equation}
 with an $R$-bound independent of $\tau$.
 Since $\Re(1-\zeta)\ge 0$ for $|\zeta|\le 1$, with strict inequality for $\zeta\ne 1$, we have that
$$
\{ M(\zeta)\,:\, |\zeta|\le 1, \zeta\ne 1 \} \subset \{ \lambda (\lambda - A)\inv \, : \,    \Re \lambda > 0 \},
$$
where the latter set is $R$-bounded by 
the ``only if" direction of Theorem~\ref{MaxR}.
Since
$$
(1-\zeta)M'(\zeta) = -M(\zeta) + M(\zeta)^2,
$$
we obtain \eqref{M-Euler-R-bounded}, with an $R$-bound independent of the stepsize $\tau$. The stated result therefore follows from Theorem~\ref{thm:Blunck}.
 \end{proof}

\begin{theorem}
\label{theorem: CN}
If 
    $A$ has maximal $L^p$-regularity, for $1<p<\infty$,
    then the numerical solution  $(u_n)_{n=1}^N$ of \eqref{C-N}, obtained by the Crank--Nicolson method with stepsize $\tau$, is bounded by    \begin{equation*}
        \big\|(\dot u_n )_{n=1}^N\big\|_{\ell^p(X)} + \big\|(A u_n )_{n=1}^N\big\|_{\ell^p(X)} \leq C_{p,X} \big\|(f_n)_{n=0}^N\big\|_{\ell^p(X)},
    \end{equation*}
    where the constant is independent of $N$ and $\tau$.
\end{theorem}
\begin{proof} 
We only have to slightly modify the previous proof. In contrast to before, now the factor $1+e^{i\theta}$ in condition \eqref{M-tilde-set} becomes important.
Using the generating functions we obtain
\begin{equation*}
    \Big(\frac{1-\zeta}{\tau}-A\frac{1+\zeta}{2}\Big) u(\zeta) = \frac{1+\zeta}2 \,f(\zeta),
\end{equation*}
which can be rewritten as
\begin{equation*}
   \Big(\frac2\tau \frac{1-\zeta}{1+\zeta}-A\Big) u(\zeta) = f(\zeta) .
\end{equation*}
Introducing $\delta(\zeta)=2 {(1-\zeta)}/{(1+\zeta)}$, we arrive at
$$
    \dot u(\zeta)= \frac{1-\zeta}{\tau}u(\zeta) = \frac{1+\zeta}2 M(\zeta)f(\zeta)
    \quad\hbox{with}\quad  M(\zeta)=   \frac{\delta(\zeta)}\tau \biggl(\frac{\delta(\zeta)}\tau-A\biggr)\inv ,
$$
and
$$
Au(\zeta) = \bigl(M(\zeta)-1\bigr) f(\zeta).
$$
To apply Theorem~\ref{thm:Blunck}, it suffices to show that the set
\begin{equation}\label{M-CN-R-bounded}
 \big\{ M(\zeta)\,:\, |\zeta|\le 1, \zeta\ne \pm 1\big\} \cup \big\{(1+\zeta)(1-\zeta)M'(\zeta)\,:\,|\zeta|\le 1, \zeta\ne {\pm} 1\big\}
 \quad\hbox{is $R$-bounded.}
 \end{equation}
For the Crank--Nicolson method,
we  have $\Re \delta(\zeta)\ge 0$ for
$|\zeta|\le 1, \zeta\ne - 1$, and $\delta(\zeta)\ne 0$ for $\zeta\ne 1$, so that
$$
\big\{ M(\zeta)\,:\, |\zeta|\le 1, \zeta\ne \pm 1 \big\} \subset \big\{ \lambda (\lambda - A)\inv \, : \,    \Re \lambda \ge 0, \lambda\ne 0 \big\},
$$
where the latter set is again $R$-bounded by 
Theorem~\ref{MaxR}. Since
$$
  (1-\zeta)(1+\zeta) M '(\zeta) = -2 M(\zeta) + {2}M(\zeta)^2,
$$
we then obtain \eqref{M-CN-R-bounded}, and hence Theorem~\ref{thm:Blunck} yields the stated result.
 \end{proof}
%
%
%

\section{Backward difference formulae}
\label{section: BDF}

We consider general $k$-step backward
difference formulae (BDF) for the discretization of \eqref{IVP}:
\begin{equation}\label{def: BDF}
    \frac{1}{\tau} \sum_{j=0}^k \delta_j u_{n-j}  =
    Au_n+f_n, \qquad n \geq k,
\end{equation}
where the coefficients of the method are given by 
\begin{equation*}
    \delta(\zeta)=\sum_{j=0}^k \delta_j \zeta^j=\sum_{\ell=1}^{k} \frac{1}{\ell}(1-\zeta)^\ell  .
\end{equation*}
The method is known to have order $k$ for $k\leq6$, and to be
A($\alpha$)-stable with angle $\alpha = 90^\circ$, $90^\circ$, $86.03^\circ$, $73.35^\circ$, $51.84^\circ$, $17.84^\circ$ for $k=1,\dots, 6$, respectively;
see \cite[Chapter~V]{HairerWannerII}. A($\alpha$)-stability is equivalent to $|\arg\delta(\zeta)|\le \pi-\alpha$ for $|\zeta|\le 1$. Note that the first and second-order BDF methods are A-stable, that is, $\Re\delta(\zeta)\ge 0$ for $|\zeta|\le 1$.
%

In this section, we use the  notation
\begin{equation*}
    \dot u_n = \frac{1}{\tau} \sum_{j=0}^k \delta_j u_{n-j}
\end{equation*}
for the approximation to the time derivative. We consider the method with zero starting values,
\begin{equation}\label{starting-from-zero}
u_0=\ldots=u_{k-1}=0.
\end{equation}
Like for the continuous problem, the effect of non-zero starting or initial values needs to be studied separately, but this is not related to the notion of maximal $L^p$- or $\ell^p$-regularity. We will discuss the case of an initial value $u_0=0$ and possibly non-zero starting values $u_1,\ldots,u_{k-1}$ in Remark~\ref{RemarkBDF}.

\subsection{BDF method of order $\bf 2$}
We obtain preservation of maximal $L^p$-regularity also for time discretization by the A-stable second-order BDF method.

\begin{theorem}
If 
    $A$ has maximal $L^p$-regularity, for $1<p<\infty$,
    then the numerical solution
    $(u_n)_{n=k}^N$ of \eqref{def: BDF} with \eqref{starting-from-zero},
    obtained by the two-step BDF method with stepsize $\tau$, is  bounded by
    \begin{equation*}
        \big\|(\dot u_n )_{n=k}^N\big\|_{\ell^p(X)} + \big\|(A u_n )_{n=k}^N\big\|_{\ell^p(X)} \leq C_{p,X}\big\|(f_n)_{n=k}^N\big\|_{\ell^p(X)},
    \end{equation*}
    where the constant is independent of $N$ and $\tau$.
\end{theorem}
\begin{proof} 
We consider the generating function of both sides
of \eqref{def: BDF}
and obtain
\begin{equation}\label{u-bdf}
  u(\zeta) = \biggl(\frac{\delta(\zeta)}\tau-A\biggr)\inv f(\zeta)
\end{equation}
so that
\begin{equation*}
   \frac {\delta(\zeta)}\tau u(\zeta) = M(\zeta) f(\zeta)
        \quad\hbox{with}\quad
        M(\zeta)=  \frac{\delta(\zeta)}\tau \biggl(\frac{\delta(\zeta)}\tau-A\biggr)\inv .
\end{equation*}
Since ${\rm Re}\,\delta(\zeta)\geq 0$
for  $|\zeta|\le 1$ (this expresses the A-stability of the method) and $\delta(\zeta)\ne 0$ for $\zeta\ne 1$,
it follows as before from Theorem~\ref{MaxR} that the set
$$
\{M(\zeta): |\zeta|\le 1, \zeta\ne 1\} \ \hbox{
is $R$-bounded. }
$$
We also have that
$$
\{(1-\zeta)M'(\zeta): |\zeta|\le 1, \zeta\ne 1\} \ \hbox{
is $R$-bounded, }
$$
because, with $\mu(\zeta)=\delta(\zeta)/(1-\zeta)= \frac12(3 -\zeta)$,
\begin{align*}
    (1-\zeta) M'(\zeta) =&\ -(1-\zeta)
     \frac{\delta'(\zeta)}\tau A \Bigl(\frac{\delta(\zeta)}\tau-A\Bigr)^{-2}   \\
    =&\ -\frac{1-\zeta}\tau  \bigl(-\mu(\zeta)
    + (1-\zeta) \mu'(\zeta) \bigr) A \Bigl(\frac{\delta(\zeta)}\tau-A\Bigr)^{-2}  \\
    =&\  \bigg( 1-(1-\zeta) \frac{\mu'(\zeta)}{\mu(\zeta)} \bigg)
    \frac{\delta(\zeta)}\tau A \Bigl(\frac{\delta(\zeta)}\tau-A\Bigr)^{-2}
    \\
    =&\  \bigg( 1-(1-\zeta) \frac{\mu'(\zeta)}{\mu(\zeta)} \bigg)
    M(\zeta)\bigl( 1 -M(\zeta) \bigr),
\end{align*}
where $(1-\zeta) \mu'(\zeta)/\mu(\zeta)$
is a bounded scalar function,
since $\mu(\zeta)\neq 0$ for $|\zeta|\le 1$. Therefore, Theorem~\ref{thm:Blunck} yields the result.
 \end{proof}

\begin{remark} The above proof extends in a direct way to yield discrete maximal $\ell^p$-regularity for A-stable linear multistep methods
$$
\sum_{j=0}^k \alpha_j u_{n+j} = \tau \sum_{j=0}^k \beta_j (Au_{n+j}+f_{n+j}), \qquad n\ge 0,
$$
that have the further property that the quotient of the generating polynomials,
$$
\delta(\zeta)= \frac{\alpha_0\zeta^k+\alpha_1\zeta^{k-1}+\ldots + \alpha_k\zeta^0}{\beta_0\zeta^k+\beta_1\zeta^{k-1}+\ldots + \beta_k\zeta^0},
$$
has no poles or zeros in the closed unit disk $|\zeta|\le 1$, with the exception of a zero at~1. We note that here A-stability is equivalent to $\Re\delta(\zeta)\ge 0$ for $|\zeta|\le 1$, and the requirement that $\delta(\zeta)$ has no pole for $|\zeta|\le 1$ is equivalent to stating that $\infty$ is an interior point of the stability region on the Riemann sphere; cf., e.g., \cite{HairerWannerII}. Note, however, that by Dahlquist's order barrier \cite{Dahlquist63}, A-stable linear multistep methods have at most order 2, and the practically used A-stable multistep methods are the second-order BDF method and the Crank--Nicolson method.
\end{remark}

\subsection{Higher order BDF methods}

We obtain maximal regularity for the BDF methods up to order 6 under a  $R$-boundedness condition in a larger sector.

\begin{theorem}\label{BDFk}
Suppose that the set $\{ \lambda(\lambda -A)\inv \,:\, |\arg\lambda|<\vartheta \}$ is $R$-bounded for an angle $\vartheta>\pi-\alpha$, where $\alpha$ is the angle of A($\alpha$)-stability of the $k$-step BDF method, for $3\le k \le 6$.    Then the numerical solution
    $(u_n)_{n=k}^N$ of \eqref{def: BDF} with \eqref{starting-from-zero},
    obtained by the $k$-step BDF method with stepsize $\tau$, is bounded by
    \begin{equation*}
        \big\|(\dot u_n )_{n=k}^N\big\|_{\ell^p(X)} + \big\|(A u_n )_{n=k}^N\big\|_{\ell^p(X)} \leq C_{p,X}\big\|(f_n)_{n=k}^N\big\|_{\ell^p(X)}
    \end{equation*}
 for $1<p<\infty$,   where the constant is independent of $N$ and $\tau$.
\end{theorem}
\begin{proof}$\,$
For the A($\alpha$)-stable $k$-step BDF method,
$|\arg\delta(\zeta)|\le \pi-\alpha<\vartheta$ for $|\zeta|\le 1$, $\zeta\ne 1$, and so the set
$$
\Bigl\{\frac{\delta(\zeta)}\tau\Bigl( \frac{\delta(\zeta)}\tau -A\Bigr)^{-1} :\,|\zeta|\le 1, \ \zeta\ne 1\Bigr\}
\! \subset
\big\{ \lambda(\lambda -A)\inv :\, |\arg\lambda|<\vartheta \big\}
$$
 is $R$-bounded, with an $R$-bound independent of $\tau$.
The rest of the proof is the same as for the two-step BDF method.
 \end{proof}

\begin{remark}\label{BDFkR}
%
According to Weis \cite[Lemma~4.c]{Weis1},
for $X=L^q$ and $q\in(1,2]$ the  set
$\{ \lambda(\lambda -A)\inv \,:\, |\arg\lambda|<\vartheta \}$ is $R$-bounded for any angle  $\vartheta<\frac\pi2 +\sigma q/2$,
where $\sigma$ is the angle of the sector where the semigroup is bounded analytic.
If $X=L^q$ and $q\in[2,\infty)$,
then a duality argument shows that
the condition is satisfied with angle $\vartheta<\frac\pi2 +\sigma q'/2$, where $1/q+1/q'=1$.
In particular, if $X=L^q$, $1<q<\infty$,
and $A=\Delta$ (the Dirichlet Laplacian),
then the $R$-boundedness condition is satisfied with any angle $\vartheta<\frac34\pi$, and
so the condition of Theorem \ref{BDFk}
is satisfied for the BDF methods of orders $1\leq k\leq 5$.
\end{remark}

\begin{remark}\label{RemarkBDF}
If $u_0=0$ but $u_1,\dots, u_{k-1}$ may not be zero,
then we define
$\widetilde f_n=f_n$ for $n\geq k$ and
\begin{equation*}
    \widetilde f_n:
    =\frac{1}{\tau} \sum_{j=0}^n \delta_j u_{n-j} - A u_n
    \quad\mbox{for}\,\,\, n=0, \dots , k-1,
\end{equation*}
so that
\begin{equation*}
    \frac{1}{\tau} \sum_{j=0}^k \delta_j u_{n-j} - A u_n
    = \widetilde f_n \quad\mbox{for}\,\,\, n\ge 0 .
\end{equation*}
Then we obtain that
\begin{align*}
     &\big\|(\dot u_n )_{n=1}^N\big\|_{\ell^p(X)}
     + \big\|(A u_n )_{n=1}^N\big\|_{\ell^p(X)} \nn
     \leq C \big\|(\widetilde f_n)_{n=1}^N\big\|_{\ell^p(X)} \nn\\
     &\leq C \big\|(f_n)_{n=k}^N\big\|_{\ell^p(X)}
     +C \bigg(\sum_{i=1}^{k-1}\big\|u_i/\tau\big\|_{X}^p\bigg)^{\frac{1}{p}}
     +C \bigg(\sum_{i=1}^{k-1}\big\|Au_i\big\|_{X}^p\bigg)^{\frac{1}{p}} ,
    \end{align*}
where the constant $C$ does not depend on $N$ and $\tau$.
    In the next section, we shall see that
    if $u_0=0$ and the starting values $u_1,\dots,u_{k-1}$
    are computed by an A-stable Runge--Kutta method with 
    invertible coefficient matrix $\AC$, then we have
\begin{align*}
\bigg(\sum_{i=1}^{k-1}\big\|u_i/\tau\big\|_{X}^p\bigg)^{\frac{1}{p}}
+\bigg(\sum_{i=1}^{k-1}\big\|Au_i\big\|_{X}^p\bigg)^{\frac{1}{p}}
\leq C_{p,X}\big\|(f_i)_{i=1}^{k-1}\big\|_{\ell^p(X)}.
\end{align*}
\end{remark}

\section{A-stable Runge--Kutta methods}
\label{section: RK}
We consider an  implicit Runge--Kutta \linebreak meth\-od with $s$ stages for the time discretization of the evolution equation \eqref{IVP}. We refer to Hairer \& Wanner \cite{HairerWannerII} for the basic notions related to such methods.

\medskip
The coefficients of the method are given by the Butcher tableau
\[
 \begin{array}{c|c}
    c \,& \,\AC \\
     \hline & \,b^T
  \end{array}
 \quad = \quad
   \begin{array}{c|c}
    (c_{i}) & (a_{ij}) \\
     \hline & (b_{j})
  \end{array}
  \qquad \quad (i,j = 1,\dotsc, s).
  \]
  Applied to the evolution equation \eqref{IVP}, a step of the method with stepsize $\tau>0$ reads
\begin{equation}
\label{eq: R-K method}
    \begin{alignedat}{3}
        U_{ni} &= u_{n} + \tau \sum_{j=1}^{s} a_{ij} \dot{U}_{nj}, \qquad &&\text{for} \quad i=1,2,\dotsc,s, \\
        u_{n+1} &= u_{n} + \tau \sum_{i=1}^{s} b_{i} \dot{U}_{ni},&& n\geq 1 \\
        \dot{U}_{ni} &= A U_{ni} + f(t_n + c_i \tau) \qquad &&\text{for} \quad i=1,2,\dotsc,s,
    \end{alignedat}
\end{equation}
Here $u_n\in X$ is the solution approximation a the $n$th time step, $U_{ni}\in X$ are the internal stages, and $\dot{U}_{ni}\in X$ is again not a continuous derivative, but a suggestive notation for the increments.

The {\it stability function} of the Runge--Kutta method is the rational function
\begin{equation*}
    R(z) = 1 + z b^T(I-z\AC)\inv \one,
\end{equation*}
where $\one=(1,\ldots,1)^T\in \R^s$. The stability function is a rational approximation to the exponential function, $R(z)=e^z+O(z^{r+1})$ for $z\to 0$, where $r$ is greater or equal to the order of the Runge--Kutta method, which we always assume to be at least $1$. Note that if $\AC$ is invertible, then $R(\infty) = 1 - b^T \AC \inv \one$.

The Runge--Kutta method is {\it A-stable} if $I-z\AC$ is nonsingular for $\Re z \leq 0$ and the stability function satisfies
$$
|R(z)|\le 1\quad\hbox{ for }\quad \Re z \leq 0.
$$

\begin{example}
Radau IIA methods are an important class of Runge--Kutta methods that are A-stable, have an invertible matrix $\AC$ and have $R(\infty)=0$ for an arbitrary number of stages $s\ge 1$; see \cite[Section~IV.5]{HairerWannerII}. For these methods, $b_j = a_{sj}$, so that $u_{n+1}=U_{ns}$. The $s$-stage method has classical order $2s-1$, that is, the error on a finite time interval is bounded by $O(\tau^{2s-1})$ when the method is applied to smooth ordinary differential equations. For  parabolic problems as considered in this paper, the order of approximation is studied in \cite{LO93} and is typically a non-integer number between $s+1$ and $2s-1$.
Radau IIA methods can be viewed as collocation methods on the Radau quadrature nodes. For linear evolution equations they can  alternatively be viewed as fully discretized discontinuous Galerkin methods with Radau quadrature on the integral terms; see \cite{AMN11}. The $s$-stage method has classical order $2s$. The $1$-stage Radau IIA method is the implicit Euler method.
\end{example}

\begin{example} Gauss methods are a class of Runge--Kutta methods that are A-stable for all stage numbers $s\ge 1$, have an invertible matrix $\AC$ and have $R(\infty)=(-1)^s$; see \cite[Section~IV.5]{HairerWannerII}. The $s$-stage method has classical order $2s$. The $1$-stage Gauss method is the implicit midpoint rule (Crank--Nicolson method).
\end{example}

%

We have the following result on discrete maximal regularity.

\begin{theorem}
\label{theorem: R-K}
{\it
Consider an A-stable Runge--Kutta method with an invertible coefficient matrix $\AC$. If the operator $A$ has maximal $L^p$-regularity, for $1<p<\infty$, then
    the numerical solution  \eqref{eq: R-K method}, obtained by the Runge--Kutta method with stepsize $\tau$, is bounded by    \begin{equation*}
        \sum_{i=1}^s \big\|(\dot U_{ni} )_{n=0}^N\big\|_{\ell^p(X)} + \sum_{i=1}^s \big\|(A U_{ni} )_{n=0}^N\big\|_{\ell^p(X)} \leq C_{p,X} \sum_{i=1}^s \big\|(f(t_n+c_i\tau))_{n=0}^N\big\|_{\ell^p(X)},
    \end{equation*}
    where the constant is independent of $N$ and $\tau$.
    }
\end{theorem}

\begin{proof} 
We use the generating functions
\begin{equation*}
    u(\zeta)=\sum_{n=0}^\infty u_n \zeta^n, \qquad U(\zeta)=\sum_{n=0}^\infty U_n \zeta^n \andquad F(\zeta)=\sum_{n=0}^\infty f_n \zeta^n,
\end{equation*}
where $U_n = (U_{ni})_{i=1}^s\in X^s$ and
$F_n = (f(t_n+c_i\tau))_{i=1}^s\in X^s$. We write $AU_n=(AU_{ni})_{i=1}^s$ and in this way consider $A$ in an obvious way as an operator on $X^s$, that is, we write $A$ instead of the Kronecker product $I_s\otimes A$ for brevity.

Following \cite{LO93}, we define the $s\times s$ matrix-valued function
\begin{equation*}
    \Delta(\zeta) = \left(\AC + \frac{\zeta}{1-\zeta}\one b^T\right)\inv,
\end{equation*}
which will be a key object in our discrete maximal regularity analysis for Runge--Kutta methods. It will play a similar role as $\delta(\zeta)$ in Sections 3 and 4, but is now matrix-valued instead of scalar-valued.
%
The formula of \cite[Lemma 2.4]{LO93},
\begin{equation}\label{Delta-z}
(\Delta(\zeta)-z)\inv = \AC(I-z\AC)\inv
+ (I-z\AC)\inv \one b^T (I-z\AC)\inv \frac \zeta {1-R(z)\zeta},
\end{equation}
shows that for a Runge--Kutta method with invertible matrix $\AC$,  the spectrum of $\Delta(\zeta)$ satisfies
    \begin{equation*}
        \sigma(\Delta(\zeta)) \subseteq \sigma(\AC\inv) \cup \big\{ z \in \C \ : \ R(z)\zeta = 1 \big\} .
    \end{equation*}
    Hence, for an A-stable method the spectrum of  $\Delta(\zeta)$ is contained in the closed right half-plane without $0$ for $|\zeta|\le 1$ with $\zeta\ne 1$, since $|R(z)|\ge 1$ requires $\Re z \ge 0$ by A-stability and $R(0)=1$.

It was shown in \cite[Proposition~2.1, equation (2.9)]{LO93} that
\begin{equation}\label{U-RK}
    U(\zeta)= \Big(\frac{\Delta(\zeta)}{\tau} - A \Big)\inv F(\zeta).
\end{equation}
Hence,
$$
\dot U(\zeta)=AU(\zeta)+F(\zeta) = M(\zeta)F(\zeta) \quad\hbox{ with }\quad
M(\zeta)= \frac{\Delta(\zeta)}{\tau}\Big(\frac{\Delta(\zeta)}{\tau} - A \Big)\inv .
$$
In view of Theorem~\ref{thm:Blunck} on the space $X^s$ instead of $X$, it suffices to prove that 
\begin{equation}\label{M-RK}
\big\{ M(\zeta)\,:\, |\zeta|\le 1, \zeta\ne 1\big\} \cup \big\{ (1+\zeta)(1-\zeta)M'(\zeta)\,:\, |\zeta|\le 1, \zeta\ne \pm1\big\}
\ \hbox{ is $R$-bounded. }
\end{equation}
We use the Cauchy-type integral formula
$$
M(\zeta)=\frac1{2\pi i} \int_\Gamma (z-\Delta(\zeta))\inv \otimes \frac z\tau \Bigl( \frac z\tau - A \Bigr)\inv \,
 \, \d z\,,
$$
where $\Gamma$ is a union of circles centered at the eigenvalues of $\Delta(\zeta)$ and lying in the sector $\Sigma_\vartheta$ of $R$-boundedness of
$\{\lambda(\lambda-A)^{-1}:\lambda\in \Sigma_\vartheta\}  $ with $\vartheta>\frac\pi2$. We also use the integral formula differentiated with respect to $\zeta$. We insert formula \eqref{Delta-z} and its derivative with respect to $\zeta$ in the integrands. The estimates required for proving \eqref{M-RK} are different in the three cases $|R(\infty)|<1$, $R(\infty)=-1$, and $R(\infty)=+1$, which in the following are studied in items (a), (b), and (c), respectively.

(a) We consider first the case where $|R(\infty)|<1$. We distinguish two situations:

(i) If $\zeta$ with $|\zeta|\le 1$ is bounded away from $1$, $|\zeta-1|\ge c >0$, then all eigenvalues of $\Delta(\zeta)$ have non-negative real part and are bounded away from $0$. Therefore the radii of all circles can be chosen to have a fixed lower bound (depending on $c>0$ and $\vartheta>\frac\pi2$), and we then have for
$|\zeta|\le 1$ with $|\zeta-1|\ge c >0$ that $|1-R(z)\zeta|\ge c'>0$ uniformly for $z$ on each circle. This yields
\begin{equation}\label{Delta-int-bounds}
\begin{aligned}
 \int_\Gamma \|(z-\Delta(\zeta))\inv \| \, |\d z| \le C,
\\
\int_\Gamma \|(1+\zeta)(1-\zeta) \frac{\partial}{\partial\zeta}(z-\Delta(\zeta))\inv\| \, |\d z| \le C,
\end{aligned}
\end{equation}
where $\|\cdot\|$ denotes an arbitrary matrix norm.

(ii) If $\zeta$ with $|\zeta|\le 1$ is close to $1$, then the implicit function theorem yields that there is a unique $z_0(\zeta)$ near $0$ with $R(z_0(\zeta))\zeta=1$, and we obtain $z_0(\zeta)=1-\zeta+O((1-\zeta)^2)$, so that for suffciently small $|1-\zeta|$ we have
$|z_0(\zeta)|\ge \frac12 |1-\zeta|$. By A-stability, we further have $\Re z_0(\zeta)\ge 0$ for $|\zeta|\le 1$. The radius $r$ of the circle in $\Sigma_\vartheta$ around $z_0(\zeta)$ can be chosen proportional to $|z_0(\zeta)|$, and hence to $|1-\zeta|$, depending on $\vartheta>\frac\pi2$. For $z$ on this circle we have
$$
1-R(z)\zeta = R(z_0(\zeta))\zeta-R(z)\zeta =
\bigl( (z_0(\zeta)-z) + O(z_0(\zeta)-z)^2) \bigr)\zeta,
$$
so that $|1-R(z)\zeta|\ge r/2$ on this circle. This yields again the bounds \eqref{Delta-int-bounds}, uniformly for $\zeta$ in a small neighbourhood of $1$ with $|\zeta|\le 1$.

We thus have proved the bounds \eqref{Delta-int-bounds} uniformly for $|\zeta|\le 1$, $\zeta\ne 1$.
By Theorem~\ref{MaxR} and Lemma~\ref{lemma: abs convex hull}, the bounds
 \eqref{Delta-int-bounds} yield \eqref{M-RK},  in the considered case where $|R(\infty)|<1$.

(b) We now consider the case $R(\infty)=-1$.  

(i) If $\zeta$ is bounded away from 
both $1$ and $-1$, the proof is the same as part (i) of (a). 

(ii) If $\zeta$ is close to $1$, the 
proof is the same as part (ii) of (a). 

(iii) If $\zeta$ is close to $-1$, we proceed as follows. A-stability and $R(\infty)=-1$ imply that
$$
R(z) = -1 -cz^{-1}  + O(z^{-2}) \quad\hbox{for } z\to\infty, \quad\hbox{ with } c>0.
$$
For $\zeta$ close to $-1$, there exists therefore a unique $z_\infty(\zeta)$ of large absolute value and with non-negative real part such that $R(z_\infty(\zeta))\zeta=1$. The Cauchy-type integrals then contain a contribution from a circle around $z_\infty(\zeta)$,
contained in $\Sigma_\vartheta$, with a radius that can be chosen proportional to $|z_\infty(\zeta)|$. The distance of this circle from the origin can also be chosen proportional to $|z_\infty(\zeta)|$. For $z$ on this circle we then have
$$
1-R(z)\zeta = R(z_\infty(\zeta))\zeta -R(z)\zeta = -c(z_\infty(\zeta)^{-1}-z^{-1})\zeta + O(z_\infty(\zeta)^{-2})
$$
and therefore $|1-R(z)\zeta|$ is bounded from below by a positive constant times
$|z_\infty(\zeta)^{-1}|$, which in turn is bounded from below by a positive constant times $|1+\zeta|$. With \eqref{Delta-z} it follows that on this circle,
$$
\| (\Delta(\zeta)-z)^{-1} \| +\| (1+\zeta)\frac{\partial}{\partial\zeta}(\Delta(\zeta)-z)^{-1} \|\le C |z|^{-2} |1+\zeta|^{-1} \le C |1+\zeta|.
$$
This yields \eqref{Delta-int-bounds} (note that the factor $1+\zeta$ in the second integral of \eqref{Delta-int-bounds} is now needed).
When $\zeta$ is away from $-1$, the proof is the same as 
part (i) of (a). 
We therefore obtain  \eqref{M-RK} also in the case $R(\infty)=-1$.

(c) The remaining case $R(\infty)=1$ can be dealt with in the same way. 
We now have
$$
R(z) = 1 +cz^{-1}  + O(z^{-2}) \quad\hbox{for } z\to\infty, \quad\hbox{ with } c>0,
$$
and the bounds \eqref{Delta-int-bounds} 
can be obtained by the same arguments 
as in the case $R(\infty)=-1$, 
replacing $1+\zeta$ by $1-\zeta$ on every occurrence.

We have thus obtained  \eqref{M-RK} for every A-stable Runge--Kutta method with invertible coefficient matrix $\AC$.
Theorem~\ref{thm:Blunck} now yields the stated result.
 \end{proof}

By \cite[Lemma~3.1]{LO93}, we have for a Runge--Kutta method with invertible coefficient matrix $\AC$ that
\begin{equation}\label{u-from-U}
        u_{n+1} = b^T \AC\inv \sum_{k=0}^n R(\infty)^{n-k}U_k,
    \end{equation}
 and so we obtain the following corollary.

\begin{corollary}
 {\it   Under the assumptions of Theorem~\ref{theorem: R-K}, and if $|R(\infty)|<1$,
we have
    \begin{equation*}
        \Big\|\Bigl( \frac{u_{n}-u_{n-1}}{\tau} \Bigr)_{n=1}^N\Big\|_{\ell^p(X)} + \big\|(A u_n )_{n=1}^N\big\|_{\ell^p(X)} \leq C_{p,X} \sum_{i=1}^s \big\|(f(t_n+c_i\tau))_{n=0}^N\big\|_{\ell^p(X)},
    \end{equation*}
    where the constant is independent of $N$ and $\tau$.
    }
\end{corollary}

\section{Space-time
full discretizations}
\label{section: full discretization}

Let $X$ be a UMD space
and let $X_h$, $h>0$, be a family of closed subspaces
of $X$ such that there exist  linear projection operators
$P_h: X\rightarrow X_h$ satisfying
\begin{align}\label{P-h}
&\|P_hu\|_{X}\leq C_0 \|u\|_X,\qquad\forall\,\, u\in X,
\end{align}
where the constant $C_0$ is independent of $h$.
Consider the problem
\begin{equation}
\label{IVPD}
    \left\{
    \begin{aligned}
        &u'_h(t) =   A_hu_h(t)+ f_h(t), \qquad  \, t>0,\\
        &u_h(0) = 0,
    \end{aligned}
    \right.
\end{equation}
where $A_h$ is the generator
of a bounded analytic semigroup
on $X_h$ and $f_h(t)$, $u_h(t)\in X_h$
for all $t>0$.
We have the following result.
\begin{theorem}
Assume \eqref{P-h} and that the collection of operators
$\{\lambda(\lambda-A_h)^{-1}: \lambda\in\Sigma_{\vartheta}\}$
is $R$-bounded in $B(X_h)$
with an $R$-bound $C_R$ that is independent of $h$. Let $\vartheta>\pi-\alpha$, where $\alpha$ is the angle of A($\alpha$)-stability of the time discretization method considered. Then, all theorems of Sections 3 to 5 hold for the numerical methods applied to \eqref{IVPD}, with constants $C_{p,X_h}$
that are independent of
both $\tau$ and $h$.
%
%
\end{theorem}

\begin{proof}
Since $\lambda(\lambda-A_h)^{-1}$ is $R$-bounded in $B(X_h)$
for $\lambda\in\Sigma_{\vartheta}$,
it follows that the collection of operators
\begin{equation}\label{R-h}
\{\lambda(\lambda-A_h)^{-1}P_h: \lambda\in\Sigma_{\vartheta}\}
\quad\mbox{is $R$-bounded in $B(X)$ }
\end{equation}
and the $R$-bound is at most $C_0C_R$.
The numerical solution given by the backward Euler scheme satisfies
\begin{align*}
\frac{u_{h,n}-u_{h,n-1}}{\tau}=A_hu_{h,n}+P_hf_{h,n},
\end{align*}
and so it follows that the generating functions are related by
\begin{align*}
\frac{1-\zeta}{\tau}u_{h}(\zeta)=M_h(\zeta)f_h(\zeta) \quad\hbox{ with }\ M_h(\zeta)=\frac{\delta(\zeta)}{\tau}
\bigg(\frac{\delta(\zeta)}{\tau}-A_h\bigg)^{-1}P_h  .
\end{align*}
In the same way as in the proof of Theorem~\ref{theorem: IE}, it is concluded from \eqref{R-h} that $M_h(\zeta)$ satisfies the $R$-boundedness condition
\eqref{M-Euler-R-bounded} with an $R$-bound that is independent of $\tau$ and $h$, and then Theorem~\ref{thm:Blunck}
yields the desired discrete maximal $\ell^p$-regularity bound, uniformly in $\tau$ and~$h$.

The results for the other methods (Crank--Nicolson,
BDF and A-stable Runge--Kutta)
are proved in the same way.
\end{proof}

\begin{remark}
If $\Omega$ is a bounded smooth domain in $\R^d$ $(d\geq 1)$,
$X=L^q(\Omega)$, $X_h$ is the standard finite element subspace
of $X$,
$A$ is a second-order elliptic partial differential operator
and $A_h$ is its finite element approximation,
then the $R$-boundedness of
$\{\lambda(\lambda-A_h)^{-1}: \lambda\in\Sigma_{\vartheta}\}$
in $B(X_h)$ has been proved in \cite{Li}
for some $\vartheta>\pi/2$ that is
independent of $h$.
The operator $P_h$ can be chosen as
the $L^2$-projection operator.
\end{remark}

\section{Logarithmically quasi-maximal $\ell^\infty$-regularity}
\label{section: ell infty est}

In this section we give some bounds that show maximal $\ell^\infty$-regularity up to a factor that is logarithmic in the number of time steps. We note that the results of this section are valid for an arbitrary complex Banach space $X$ (not necessarily a UMD space as in the previous sections), and $R$-boundedness plays no role in this section. We just assume that $A$ is the generator of analytic semigroup on $X$, and
$\lambda (\lambda -A)^{-1}$ is uniformly bounded for $\lambda\in\Sigma_\vartheta$ with an angle $\vartheta>\pi-\alpha$ for the angle $\alpha$ of A($\alpha$)-stability of the numerical method. We consider again  $k$-step BDF methods with $k\le 6$ and A-stable Runge-Kutta methods with an invertible coefficient matrix and $|R(\infty)|<1$.

We start with the $k$-step BDF method,
with initial condition $u_0=u_1=\ldots=u_{k-1}=0$ as in Section~\ref{section: BDF}.
By \eqref{u-bdf}, the numerical solution
can be expressed as a discrete convolution
\begin{equation}\label{Awn-m}
u_n = \tau \sum_{j=k}^n e_{n-j}(\tau A) f_j , \qquad n\ge k,
\end{equation}
with the generating function
$$
\tau \sum_{n=0}^\infty e_n(\tau A) \zeta^n = \biggl(\frac{\delta(\zeta)}\tau-A\biggr)\inv.
$$
This can be viewed as a convolution quadrature
approximation of the exact solution at $t_n=n\tau$,
\begin{equation*}
u(t_n) =\int_0^{t_n} e^{(t_n-t)A} f(t)\,\d t  .
\end{equation*}
Theorem 2.1 in \cite[Theorem 2.1]{Lubich-cqrevisited} (used with $K(\lambda)=(\lambda-A)^{-1}$ and then with $K(\lambda)=A(\lambda-A)^{-1}$) shows that
$$
\| e_n(\tau A) - e^{n\tau A} \|_{B(X)} \le Ct_{n+1}^{-k} \tau^k, \qquad n\ge 0,
$$
and
$$
\| Ae_n(\tau A) - Ae^{n\tau A} \|_{B(X)} \le Ct_{n}^{-1-k} \tau^k, \quad n\ge 1,
\qquad\hbox{and }\  \| Ae_0(\tau A)\| \le C\tau^{-1}.
$$
Since $\| Ae^{t A} \|_{B(X)} \le Ct^{-1}$ for $t>0$,
a direct consequence of the latter estimate is the following.
\begin{lemma}
\label{BDFwg}
    Suppose that $A(\lambda-A)\inv$ is uniformly bounded for $\lambda \in \Sigma_\vartheta$ with an angle $\vartheta>\pi-\alpha$ for the angle of A($\alpha$)-stability of  the $k$-step BDF method, for $1\le k \le 6$. Then we have
    \begin{equation*}
        \|Ae_n(\tau A)\|_{B(X)} \leq C/ t_{n+1},  \qquad\  n\geq 0 .
    \end{equation*}
\end{lemma}
\indent Using \refe{Awn-m} and Lemma \ref{BDFwg},
we obtain immediately the following $\ell^\infty$-bound, or more generally $\ell^p$-bound uniformly for $1\le p \le \infty$.
\begin{theorem}\label{BDFLinfty}
Suppose that $A(\lambda-A)\inv$ is uniformly bounded for $\lambda \in \Sigma_\vartheta$ with an angle $\vartheta>\pi-\alpha$ for the angle of A($\alpha$)-stability of  the $k$-step BDF method, for $1\le k \le 6$.    Then the numerical solution
    $(u_n)_{n=k}^N$ of \eqref{def: BDF} with \eqref{starting-from-zero},
    obtained by the $k$-step BDF method with stepsize $\tau$, is bounded by
\begin{align*}
    \|(Au_n)_{n=1}^N\|_{\ell^p(X)}
    &\leq  C\,  \log N\,\|(f_n)_{n=1}^N\|_{\ell^p(X)}   ,
\end{align*}
where the constant $C$ is independent of $N$ and $\tau$ and $1\le p \le \infty$.
\end{theorem}

We now turn to A-stable Runge-Kutta methods. By \eqref{U-RK}, the vector of internal stages $U_n=(U_{ni})_{i=1}^s\in X^s$
can be expressed in terms of the vector of inhomogeneity values used in the $n$th step, $F_n=(f(t_n+c_i\tau))_{i=1}^s$, as a discrete block convolution
\begin{equation}\label{AWn-m}
U_n = \tau \sum_{j=0}^n E_{n-j}(\tau A) F_j \, ,
\end{equation}
with the generating function
$$
\tau \sum_{n=0}^\infty E_n(\tau A) \zeta^n = \biggl(\frac{\Delta(\zeta)}\tau-A\biggr)\inv.
$$
We have the following lemma.
\begin{lemma}
\label{RKQW}  Suppose that $A(\lambda-A)\inv$ is uniformly bounded for $\lambda \in \Sigma_\vartheta$ with an angle $\vartheta>\pi/2$.
    For an A-stable Runge--Kutta method with
    invertible coefficient matrix $\AC$ and $|R(\infty)|<1$,
    we have
    \begin{equation*}
        \|AE_n(\tau A)\|_{B(X^s)} \leq C/t_{n+1} ,\qquad n\geq 0  .
    \end{equation*}
\end{lemma}
\begin{proof}
We use the Cauchy-type integral formula
$$
\biggl(\frac{\Delta(\zeta)}\tau-A\biggr)\inv =
\frac1{2\pi i} \int_\Gamma (z-\Delta(\zeta))\inv \otimes  \Bigl( \frac z\tau - A \Bigr)\inv \,
 \, \d z
$$
with a keyhole contour $\Gamma=\Gamma_1\cup \Gamma_2$
composed of
$$\Gamma_1=\{re^{\pm i\vartheta}: r\ge \eps \}
\quad\mbox{and}\quad
\Gamma_2= \{\eps e^{i\phi}: |\phi|\leq \vartheta\}
$$
with a small $\eps>0$.
On inserting \eqref{Delta-z}, using the geometric series for $(1-R(z)\zeta)^{-1}= \sum_{n=0}^\infty R(z)^n\zeta^n$ and collecting equal powers of $\zeta$ on the left and right-hand sides, we find
 \begin{equation*}
  \tau E_0(\tau A) = \frac{1}{2\pi\iu} \int_{\Gamma}
    \AC(I-z\AC)\inv  \otimes \Bigl( \frac z\tau - A \Bigr)\inv \d z  =
    \tau   \AC(I-\AC\otimes\tau A)\inv
\end{equation*}
and
\begin{equation}\label{intEn}
    \tau E_n(\tau A)  = \frac{1}{2\pi\iu} \int_{\Gamma}
    R(z)^{n-1} (I-z\AC)\inv \one b^T (I-z\AC)\inv  \otimes \Bigl( \frac z\tau - A \Bigr)\inv\d z ,
    \quad\ n\geq1 .
\end{equation}
Since the eigenvalues of $\AC$ have positive real part, we obtain
$$
\| \tau AE_0(\tau A) \|_{B(X^s)} \le C.
$$
Next, we estimate $\tau AE_n(\tau A)$.
Since the stability function $R(z)$
satisfies $R(z)=e^z+O(z^2)$, for sufficiently small $c$ we have
\begin{equation*}
    |R(z)| \leq e^{-\Re z/2} \quad
    \mbox{for}\,\,\, |\arg z|=\vartheta  \,\,\,\mbox{and}\,\,\, 0\leq |z|\leq c ,
    \end{equation*}
and for $z\in \Gamma$ with $|z|\geq c$ we have
$|R(z)|\leq \rho<1$. Then  \refe{intEn} yields, on applying the operator $A$, letting $\eps\to 0$ in the definition of the contour $\Gamma$ and then taking norms,
$$
     \|\tau AE_n(\tau A)\|_{B(X^s)}
    \leq C \int_0^\infty \frac{e^{-(n-1)r |\cos(\vartheta)|/2}+\rho^n} {1+r^2}\,\d r
 \leq C/n ,\qquad
 n\geq 1 .
$$
This completes the proof of Lemma \ref{RKQW}.
 \end{proof}

The identity \refe{AWn-m}  and Lemma \ref{RKQW}, and formula  \eqref{u-from-U} imply the following result.

\begin{theorem}
\label{RKLinfty}
{\it Suppose that $A(\lambda-A)\inv$ is uniformly bounded for $\lambda \in \Sigma_\vartheta$ with an angle $\vartheta>\pi/2$.
For an A-stable Runge--Kutta method with
invertible Runge--Kutta matrix $\AC$ and $|R(\infty)|<1$,
the numerical solution  \eqref{eq: R-K method} is bounded by
\begin{align*}
    \|(Au_{n+1})_{n=0}^{N-1}\|_{\ell^p(X)}  + \|(AU_n)_{n=0}^{N-1}\|_{\ell^p(X^s)} \leq &\ C\, \log N  \,\|(F_n)_{n=0}^{N-1}\|_{\ell^p(X^s)} , 
\end{align*}
where the constant $C$ is independent 
of $N$ and $\tau$ and $1\le p \le \infty$.
}
\end{theorem}

A similar logarithmically quasi-maximal regularity result
was proved in \cite{LB15} for the discontinuous
Galerkin (DG) solutions of the heat equation
with an extra logarithmic factor:
\begin{align*}
\|\partial_tu_\tau\|_{L^p(0,T;L^q)}
+\bigg(\sum\tau \bigg\|\frac{[u_\tau]}{\tau}\bigg\|_{L^q}^p\bigg)^{\frac{1}{p}}
+\|\Delta u_\tau\|_{L^p(0,T;L^q)}
\leq
C\ln\bigg(\frac{T}{\tau}\bigg)\|f\|_{L^p(0,T;L^q)},
\end{align*}
where $u_\tau$ denotes the DG solution of the heat equation,
$\partial_tu_\tau$ denotes the piecewise time derivative
of $u_\tau$, and the summation extends over all
jumps in the time interval $[0,T]$.  The discontinuous Galerkin method is closely related to the Radau IIA implicit Runge--Kutta method, but the proof given in \cite{LB15} is very different from the proof of Theorem~\ref{RKLinfty}.

\section{An application of the
discrete maximal $L^p$-regularity}
\label{section: applications}

In this section, we illustrate how to apply
the discrete maximal $L^p$-regularity
to derive error estimates and regularity uniform
in the stepsize $\tau$ of the time-discrete solution
for nonlinear parabolic equations.
In this process, we shall see the
superiority of the maximal $L^p$-regularity approach over
the widely used $L^2$-norm approach for strongly nonlinear problems.

We illustrate our idea by considering
the nonlinear parabolic equation
\begin{align}
&\frac{\partial u}{\partial t}-\Delta u=f(u,\nabla u)
&&\mbox{in}\,\,\,\Omega, \label{HM1}\\
&\frac{\partial u}{\partial \nu}=0
&&\mbox{on}\,\,\,\partial\Omega, \label{HM2}\\[3pt]
&u=u_0
&&\mbox{at}\,\,\,t=0 , \label{HM3}
\end{align}
on a smooth bounded domain $\Omega\subset\R^d$, where ${\partial u}/{\partial \nu}$ denotes the normal derivative on the
boundary $\partial\Omega$. We assume that $f:\R\times\R^d\to \R$ is a smooth pointwise nonlinearity, appearing as $f(u(x,t),\nabla u(x,t))$ in \eqref{HM1}.
Examples of such equations include
Burgers' equation (where $f=u\cdot\nabla u$)
and the harmonic map heat flow (where
$f=u|\nabla u|^2$). We will assume that this problem has a sufficiently regular solution, but we will not impose growth conditions on the nonlinearity $f$.

We consider time discretization by the backward Euler scheme
\begin{align}
&\frac{u_n-u_{n-1}}{\tau}-\Delta u_n=f(u_n,\nabla u_n)
&&\mbox{in}\,\,\,\Omega, && n\geq 1, \label{DHM1}\\
&\frac{\partial u_n}{\partial \nu}=0
&&\mbox{on}\,\,\,\partial\Omega, &&
n\geq 1, \label{DHM2} \\[3pt]
&\hbox{with starting value } u_0 . \label{DHM3}
\end{align}

Extensions to full space-time discretizations are
also discussed. Since the extension to higher-order time discretization methods satisfying maximal regularity estimates is straightforward, we just consider the backward Euler method for simplicity of presentation.

In the following, for any sequence
$v=(v_n)_{n=1}^N$ of functions in $L^q(\Omega)$ and a given stepsize $\tau>0$ we consider the scaled $\ell^p$-norm
$$
\|v\|_{L^p(L^q)}
=\bigg(\sum_{n=1}^N\tau \|v_n\|_{L^q}^p\bigg)^{\frac{1}{p}} ,
$$
for $1\leq p<\infty$, which is the
$L^p(0,N\tau;L^q(\Omega))$-norm of the
piecewise constant function
that takes the value $v_n$ on
$(t_{n-1},t_n)$. 
We write similarly
$
\displaystyle
\|v\|_{L^\infty(L^q)} = \max_{1\le n\le N} \|v_n\|_{L^q}.
$

\begin{theorem}
If the nonlinearity $f:\R\times \R^d\to\R$ is continuously differentiable
(here we do not assume any growth condition), and if the exact solution of \refe{HM1}-\refe{HM3}
satisfies $\partial_{tt}u\in L^p(0,T;L^p)$
and $u\in L^p(0,T;W^{2,p})$ for some $T>0$ and for some $p$ with $2+d<p<\infty$,
then there exist $\tau_0>0$ and $C_0>0$ (which depend on $T$)
such that for $0<\tau\le\tau_0$ and $N\tau\le T$, the errors
$$
e_n=u_n-u(\cdot,t_n)
\quad\hbox{ and }\quad\dot e_n= \frac{e_n-e_{n-1}}\tau
$$
of the time-discrete solution given by
\refe{DHM1}-\refe{DHM3} are bounded by
\begin{align}\label{LpW2p}
&\|(\dot e_{n})_{n=1}^N\|_{L^p(L^p)}+\|(\Delta e_n)_{n=1}^N\|_{L^p(L^p)}
\leq C_0\tau ,
\\[1mm]
 \label{LinfW1inf}
& \|(e_n)_{n=1}^N\|_{L^\infty(W^{1,\infty})}
\leq C_0\tau .
\end{align}
\end{theorem}

\begin{proof}
We rewrite the equations \refe{HM1}-\refe{HM3} for the exact solution $u(t)=u(\cdot,t)$ as
\begin{align}
&\frac{u(t_n)-u(t_{n-1})}{\tau}-\Delta u(t_n)=f(u(t_n),\nabla u(t_n))
+d_n
&&\mbox{in}\,\,\,\Omega, && n\geq 1, \label{THM1}\\
&\frac{\partial u(t_n)}{\partial\nu}=0
&&\mbox{on}\,\,\,\partial\Omega, &&
n\geq 1, \label{THM2} \\[3pt]
&u(t_0)=u_0 , \label{THM3}
\end{align}
where the defect $d_n=(u(t_n)-u(t_{n-1}))/\tau-\partial_t u(t_n)$
is the truncation error due to the
time discretization, satisfying
\begin{align*}
\|(d_n)_{n=1}^N\|_{L^p(L^p)}
&\leq C\|\partial_{tt}u\|_{L^p(0,T;L^p)}\,\tau .
\end{align*}

Comparing \refe{DHM1}-\refe{DHM3}
with \refe{THM1}-\refe{THM3}, we see that
the error $e_n=u_n-u(t_n)$ satisfies
\begin{align}
&\dot e_n-\Delta e_n=f(u_n,\nabla u_n)
-f(u(t_n),\nabla u(t_n)) - d_n
&&\mbox{in}\,\,\,\Omega, && n\geq 1, \label{EHM1}
\\
&\frac{\partial e_n}{\partial \nu}=0
&&\mbox{on}\,\,\,\partial\Omega, &&
n\geq 1, \label{EHM2} 
\\[3pt] 
&e_0=0 . \label{EHM3}
\end{align}
Let $M=\|u\|_{L^\infty(0,T;W^{1,\infty})}$
and define the function
$$
\rho(s)=\sup_{\begin{subarray}{ll}
|y|\leq s\\
|z|\leq s\\
x\in\Omega
\end{subarray}}
\bigg(\bigg|\frac{\partial f}{\partial y}(y,z,x)\bigg|
+\bigg|\frac{\partial f}{\partial z}(y,z,x)\bigg|  \bigg)
$$
for $s>0$.
Since the Neumann Laplacian $\Delta$
has maximal $L^p$-regularity,
Theorem \ref{theorem: IE} implies 
that $e=(e_n)_{n=1}^N$ is bounded by
\begin{align*}
&\|\dot e \|_{L^p(L^p)}+\|\Delta e\|_{L^p(L^p)} \nn\\
&\leq C\|\bigl( f(u_n,\nabla u_n)
-f(u(t_n),\nabla u(t_n)\bigr)_{n=1}^N\|_{L^p(L^p)}
+C\|d\|_{L^p(L^p)} \nn\\
&\leq C\rho(M+\|e\|_{L^\infty(W^{1,\infty})})
\|e\|_{L^p(W^{1,p})}
+C\tau  ,
\end{align*}
where we further estimate
$$
\|e\|_{L^p(W^{1,p})}  \le \epsilon \|e\|_{L^p(W^{2,p})}  + C_\epsilon \|e\|_{L^p(L^p)}.
$$
Suppose that
\begin{align}\label{LinfW}
\|e\|_{L^\infty(W^{1,\infty})}\leq 1 .
\end{align}
Then by choosing $\epsilon$ small enough the last
three inequalities imply
\begin{align*}
&\|\dot e\|_{L^p(L^p)}+\|\Delta e\|_{L^p(L^p)}
\leq C\|e\|_{L^p(L^p)}+C\tau .
\end{align*}
Since
$$
\max_{1\leq k\leq n}\|e_k\|_{L^p}
\leq \sum_{k=1}^n\tau \|\dot e_k\|_{L^p}
= \|\dot e\|_{L^1(L^p)}  \le T^{1-1/p}  \|\dot e\|_{L^p(L^p)}
$$
it follows that
\begin{align*}
\|e\|_{L^\infty(L^p)}+\|\Delta e\|_{L^p(L^p)}
&\leq C\|e\|_{L^p(L^p)}+C\tau \nn\\
&\leq \epsilon\|e\|_{L^\infty(L^p)}
+C_\epsilon\|e\|_{L^1(L^p)}+C\tau ,
\end{align*}
Since this holds for every $N$ with $N\tau\le T$,  we derive by Gronwall's inequality that\\
$
\|e\|_{L^\infty(L^p)} \le C_T\,\tau,
$
which then yields \eqref{LpW2p}.
This also implies
\begin{align*}
&\|\widetilde e\|_{W^{1,p}(0,T;L^p)}
+\|\widetilde e\|_{L^p(0,T;W^{2,p})}
\leq C\tau ,
\end{align*}
where $\widetilde e$ is the piecewise linear
interpolation of $e_n$, $n=1,\dots,N$, at the times~$t_n$.
Since
\begin{align*}
W^{1,p}(0,T;L^p)\cap L^p(0,T;W^{2,p})
&\hookrightarrow
W^{1-\theta,p}(0,T;W^{2\theta,p}) \\
&\hookrightarrow
C^{1-\theta-1/p}([0,T];W^{2\theta-d/p,\infty})
\end{align*}
for any $\theta\in (0,1)$,
when $p>2+d$ there exists $\theta\in(0,1)$
such that $$1/2+d/(2p)<\theta<1-1/p$$ and so
$W^{1,p}(0,T;L^p)\cap L^p(0,T;W^{2,p})
\hookrightarrow
L^\infty(0,T;W^{1,\infty})$. Hence, we have \eqref{LinfW1inf}.
Overall, from \refe{LinfW} one can derive
\refe{LinfW1inf}.
Therefore, by a fixed point argument one readily obtains that
there exists a positive constant $\tau_0$ such that
when $\tau<\tau_0$ we have
\refe{LpW2p}--\refe{LinfW1inf},
without assuming \refe{LinfW}.
This completes the proof of the theorem. 
\end{proof}

\begin{remark}\rm
The key argument of the above proof is that $\tau$-uniform discrete maximal $\ell^p$-regularity allows us to control the ${L^\infty(W^{1,\infty})} $-norm of the error, and hence of the numerical solution. In contrast, the logarithmically quasi-maximal $\ell^\infty$-regularity bounds of Section~\ref{section: ell infty est} are not sufficient to control the ${L^\infty(W^{1,\infty})} $-norm of the numerical solution uniformly in $\tau$ on bounded time intervals, because the logarithmic factor harms the use of the Gronwall inequality. The $\ell^\infty$-regularity bounds of Section~\ref{section: ell infty est} can, however, be used to refine the error bounds. Since we know already that \eqref{LinfW1inf} holds, we obtain from Theorem~\ref{BDFLinfty} on $X=C(\bar\Omega)$ applied to the error equations \eqref{EHM1}--\eqref{EHM3} that
$$
\| \dot e \|_{L^\infty(L^\infty)} + \| \Delta e \|_{L^\infty(L^\infty)} 
\le C' \, \log N\, \bigl(\rho (M+1) \| e \|_{L^\infty(W^{1,\infty})}+ \| d \|_{L^\infty(L^\infty)}\bigr),
$$
which directly yields, under the additional condition that $u \in C^2([0,T],C(\bar\Omega))$,
\begin{equation}
 \| (\dot e_n)_{n=1}^N \|_{L^\infty(L^\infty)} + \| (\Delta e_n)_{n=1}^N \|_{L^\infty(L^\infty)} 
\le C \tau\, \log N.
\end{equation}
\end{remark}

\begin{remark}\rm
Uniform regularity estimates such
as \refe{LpW2p} have important applications in
error estimates of  full
discretizations, with finite element  methods for the spatial discretization.
In the following, let us denote for brevity $u_h^\tau=(u_{h,n})_{n=1}^N$ the fully discrete numerical solution, $u^\tau=(u_{n})_{n=1}^N$ the result of the implicit Euler time discretization given by \refe{DHM1}-\refe{DHM3},
and $u=(u(t_n))_{n=1}^N$ the sequence of exact solution values of the nonlinear parabolic problem \refe{HM1}-\refe{HM3}.
Typically, in order to avoid any grid-ratio
condition in deriving the error estimates,
the error of the fully discrete method
can be decomposed into two parts
\cite{LS13}:
\begin{align*}
\|u_h^\tau - u \|_{L^p(W^{1,p})}
\leq \|u_h^\tau - u^\tau\|_{L^p(W^{1,p})}
+\|u^\tau - u\|_{L^p(W^{1,p})},
\end{align*}
where the first part is expected to be $O(h)$,
uniformly in $\tau$.
For such nonlinear problems as \refe{HM1}-\refe{HM3},
the main difficulty in the error estimates is
to prove the boundedness
$\|u_h^\tau\|_{L^\infty(W^{1,\infty})}\le C$
for the numerical solution.
Let $I_h$ denote the Lagrange interpolation operator.
Under the regularity of \refe{LpW2p},
the first part of the error can be proved
in the following way:
by assuming that
\begin{align} \label{UhWinf}
\|I_hu^\tau-u_h^\tau\|_{L^\infty(W^{1,\infty})} \leq 1,
\end{align}
via $\tau$- and $h$-uniform discrete maximal $\ell^p$-regularity estimates on the finite element space one can prove the $\tau$-independent error estimate (with $D_\tau$ denoting the backward difference quotient operator
and $W^{-1,p}$ denoting the dual space of $W^{1,p'}$)
\begin{align}\label{ErrW1p}
&\|D_\tau(I_hu^\tau-u_h^\tau)\|_{L^p(W^{-1,p})}
+\|I_hu^\tau-u_h^\tau\|_{L^p(W^{1,p})}
\leq Ch ,
\end{align}
and by using the inverse inequality
\begin{align*}
&\|I_hu^\tau-u_h^\tau\|_{L^p(W^{1,\infty})}
\leq Ch^{-d/p}\| I_hu^\tau-u_h^\tau\|_{L^p(W^{1,p})}
\leq Ch^{1-d/p}, \\
&\|D_\tau(I_hu^\tau-u_h^\tau)\|_{L^p(W^{1,\infty})}
\leq Ch^{-2-d/p}\|D_\tau(I_hu^\tau-u_h^\tau)\|_{L^p(W^{-1,p})}
\leq Ch^{-1-d/p} ,
\end{align*}
one recovers a better $L^\infty(W^{1,\infty})$-estimate
via the interpolation inequality
\begin{align} \nonumber
\|I_hu^\tau-u_h^\tau\|_{L^\infty(W^{1,\infty})}
&\leq
\|I_hu^\tau-u_h^\tau\|_{L^p(W^{1,\infty})}^{1-1/p}
\|D_\tau(I_hu^\tau-u_h^\tau)\|_{L^p(W^{1,\infty})}^{1/p}
\\ \label{ErrW1p2}
&\leq Ch^{1-(2+d)/p} .
\end{align}
When $p>2+d$, one can conclude that
there exists a positive constant
$h_0>0$ such that when $h<h_0$
the inequalities \refe{ErrW1p}-\refe{ErrW1p2} hold,
without pre-assuming \refe{UhWinf}.
\end{remark}

\begin{remark}\rm
We mention that
the often used $l^\infty(L^2)$-norm approach does not work
when the nonlinearity is strong enough.
Specifically,
if one uses the standard $l^\infty(L^2)$-norm error estimate,
then by assuming \refe{UhWinf} one can only prove
\begin{align*}
&\| I_hu^\tau-u_h^\tau\|_{L^\infty(L^{2})}
+h \| I_hu^\tau-u_h^\tau \|_{L^\infty(H^{1})}
\leq Ch^2
\end{align*}
for the linear finite element method.
The $L^\infty(W^{1,\infty})$ error of the numerical solution
cannot be recovered for $d\geq 2$:
\begin{align*}
&\| I_hu^\tau-u_h^\tau \|_{L^\infty(W^{1,\infty})}
\leq Ch^{-d/2}\| I_hu^\tau-u_h^\tau \|_{L^\infty(H^1)}
\leq Ch^{1-d/2} .
\end{align*}
This shows an advantage of the maximal $L^p$-regularity
approach for the analysis of strongly nonlinear problems.

Of course, if the nonlinearity is not strong,
then one only needs to assume
\refe{UhWinf} with some $L^\infty(W^{1,q})$ norm,
\begin{align*}
&\|I_hu^\tau-u_h^\tau\|_{L^\infty(W^{1,q})}
\leq Ch^{d/q-d/2}\|I_hu^\tau-u_h^\tau\|_{L^\infty(H^1)}
\leq Ch^{1+d/q-d/2}   ,
\end{align*}
and this weaker norm can thus be recovered
if $q<2d/(d-2)$.
\end{remark}

\begin{remark} \rm
Since the approach via discrete maximum $\ell^p$-regularity allows us to control the $\ell^\infty(W^{1,\infty})$ error of the numerical solution, it works equally well for nonlinearities $f(u,\nabla u)$ that are defined only in a subregion of $\R\times\R^d$, provided that the exact solution of the parabolic problem stays in that subregion. For example, this includes nonlinearities with singularities (e.g., rational functions) or functions that are defined only for positive $u$ or for
$\nabla u$ in a cone.
\end{remark}

\section*{Acknowledgement} The research stay of
Buyang Li at the University of T\"ubingen is
funded by the Alexander von Humboldt Foundation.
The work of Bal\'azs Kov\'acs is funded by
Deutsche Forschungsgemeinschaft, SFB 1173.

\pagebreak[3]

\bibliographystyle{alpha}

\end{document}